\documentclass[secthm,seceqn,amsthm,ussrhead,reqno, 12pt]{amsart}
\usepackage[utf8]{inputenc}
\usepackage[english]{babel}
\usepackage[symbol]{footmisc}
\usepackage{amssymb,amsmath,amsthm,amsfonts,xcolor,enumerate,hyperref,comment,longtable,cleveref}

\usepackage{times}
\usepackage{cite}
\usepackage{pdflscape}
\usepackage{ulem}

\usepackage{tikz}
\usepackage{hyperref}
\usepackage{cancel}
\usepackage{stmaryrd}

\newtheorem{theorem}{Theorem}

\newtheorem{corollary}[theorem]{Corollary}

\newtheorem{lemma}[theorem]{Lemma}

\newtheorem{definition}[theorem]{Definition}

\crefrangeformat{equation}{#3(#1)#4--#5(#2)#6}

\crefname{theorem}{theorem}{Theorems}
\crefname{lem}{Lemma}{Lemmas}
\crefname{cor}{Corollary}{Corollaries}
\crefname{prop}{Proposition}{Propositions}
\crefname{proposition}{Proposition}{Propositions}
\crefname{defn}{Definition}{Definitions}
\crefname{exm}{Example}{Examples}
\crefname{rem}{Remark}{Remarks}
\crefname{section}{Section}{Sections}
\crefname{equation}{\unskip}{\unskip}
\crefname{enumi}{\unskip}{\unskip}

\newcommand{\red}[1]{\textcolor{black}{#1}}

\usepackage{fouriernc}

\begin{document}

\noindent{\Large
Transposed Poisson structures on the extended Schr\"{o}dinger-Virasoro and the original deformative Schr\"{o}dinger-Virasoro algebras}

	\bigskip
	
	 \bigskip

\begin{center}	
	{\bf
Zarina Shermatova\footnote{Kimyo International University in Tashkent; V.I.Romanovskiy Institute of Mathematics Academy of Science of Uzbekistan;  \
z.shermatova@mathinst.uz}}
\end{center}

 \bigskip

\noindent {\bf Abstract.}
{\it
We compute   $\frac{1}{2}$-derivations on the extended Schr\"{o}dinger-Virasoro
 \footnote{We refer the notion of \textquotedblleft Witt\textquotedblright  \   algebra to the simple Witt algebra and the notion of \textquotedblleft Virasoro\textquotedblright  \  algebra to the central extension of the simple Witt algebra.}
 algebras and the  original deformative  Schr\"{o}dinger-Virasoro algebras.  The extended Schr\"{o}dinger-Virasoro  algebras have neither nontrivial  $\frac{1}{2}$-derivations nor nontrivial transposed Poisson algebra structures. We demonstrate that the  original deformative  Schr\"{o}dinger-Virasoro algebras have nontrivial  $\frac{1}{2}$-derivations, indicating that they possess  nontrivial  transposed Poisson structures.

}



 \bigskip
\section*{Introduction}

Poisson algebras appear in a variety of geometric and algebraic contexts, including Poisson manifolds, 
algebraic geometry, 
noncommutative geometry, 
operads, 
quantization theory, 
quantum groups,
etc. The study of Poisson algebras also led to other algebraic structures, such as generic Poisson algebras,
algebras of Jordan brackets and generalized Poisson algebras,
Gerstenhaber algebras,
Novikov-Poisson algebras,
Quiver Poisson algebras,
etc. In the recent paper \cite{bai20}, the authors initiated a study of a notion of a transposed Poisson algebra by reversing the roles of the two operations in the Leibniz rule that defines a Poisson algebra. A transposed Poisson algebra defined in this manner not only retains some characteristics of a Poisson algebra, such as closure under tensor products and Koszul self-duality as an operad, but also encompasses a diverse range of identities \cite{kms,fer23,lb23,bl23}. It is noteworthy that the authors provided several constructions of  transposed Poisson algebras from Novikov-Poisson algebras, commutative associative algebras and pre-Lie algebras \cite{bai20}. Moreover, transposed Poisson algebras are related to weak Leibniz algebras \cite{dzhuma}.

Let $R_z$ (resp., $L_z$) denote the operator of the right (resp., left) multiplication by an element $z\in L.$ We see from definition of transposed Poisson algebra that  both  $R_z$ and  $L_z$	are  $\frac{1}{2}$-derivations on Lie algebra. Actually, $R_z=L_z$ for all $z\in L.$
This motivated to define
 all transposed Poisson structures
	on  Witt and Virasoro algebras in  \cite{FKL};
	on   twisted Heisenberg-Virasoro,   Schr\"odinger-Witt  and
	extended Schr\"odinger-Witt algebras in \cite{yh21};
	on Schr\"odinger algebra in $(n+1)$-dimensional space-time in \cite{ytk}; on the
$n$-th Schr\"odinger algebra in \cite{wan};
\red{on solvable Lie algebra with filiform nilradical} in \cite{aae23};
on oscillator Lie algebras in \cite{kkh24};
	on Witt type Lie algebras in \cite{kk23};
	on generalized Witt algebras in \cite{kkg23}, \cite{kkz1}; on Virasoro-type  algebras in \cite{kkz2}; on loop Heisenberg-Virasoro  algebras \cite{yan};
 Block Lie algebras in \cite{kk22,kkg23}
  and
on    Lie incidence algebras (for all references, see the survey \cite{k23}).


 The characterization of transposed Poisson structures derived on the Witt algebra \cite{FKL} raises the question of identifying algebras related to the Witt algebra that possess nontrivial transposed Poisson structures. Consequently, several algebras associated with the Witt algebra are examined in prior works \cite{kk23,kk22,kkg23}. This paper extends that line of research. Specifically, we detail transposed Poisson structures on central extensions of the extended  Schr\"{o}dinger-Witt algebras  and   the original deformative Schr\"{o}dinger-Witt algebras.
The Schr\"{o}dinger-Witt algebra $\mathfrak{so},$  originally introduced  by Henkel \cite{h94} during his study on the invariance of the free  Schr\"{o}dinger equation, is a vector space over the complex field $\mathbb{C}$ with a basis $\{L_n, M_n, Y_{n+\frac{1}{2}} \ | \ n \in \mathbb{Z}\}$ satisfying the following non-vanishing relations
\begin{longtable}{lllll}
$[L_m,L_n]$&$=$&$(n-m)L_{m+n},$\\
$[L_m,M_n]$&$=$&$nM_{m+n},$ \\
$[L_m,Y_{n+\frac{1}{2}}]$&$=$&$\left(n+\frac{1-m}{2}\right)Y_{m+n+\frac{1}{2}},$\\
$[Y_{m+\frac{1}{2}},Y_{n+\frac{1}{2}}]$&$=$&$(m-n)M_{m+n+1}.$\\
\end{longtable}

It is easy to see that  $\mathfrak{so}$ is a semi-direct product of the Witt algebra and the two-step nilpotent infinite-dimensional Lie algebra. The structure and representation theory of  $\mathfrak{so}$ have been extensively studied by Roger and Unterberger \cite{ru06}. In order to investigate vertex representations of $\mathfrak{so},$
Unterberger \cite{unt} introduced a class of new infinite-dimensional Lie algebras  $\widetilde{\mathfrak{so}}$ called the extended Schr\"{o}dinger-Witt algebra, which can be viewed as an extension of $\mathfrak{so}$ by conformal current with conformal weight 1.
In \cite{gwp}, authors studied the derivations, the central extensions and the automorphism group of
the extended  Schr\"{o}dinger-Witt  algebra. In \cite{yua}, Lie bialgebra structure on the extended
 Schr\"{o}dinger-Witt  algebra was obtained. The notion of $n$-derivation of the extended Schr\"{o}dinger-Witt  algebra was investigated in \cite{zho}, and the main result when $n = 2$ was applied to characterize the
linear commuting maps and the commutative post-Lie algebra structures on  $\widetilde{\mathfrak{so}}.$

 Both original   and twisted   Schr\"{o}dinger-Witt algebras, and also their deformations were introduced by Henkel \cite{h94},  Unterberger \cite{hu} and Roger \cite{ru06}, in the context of non-equilibrium statistical physics, closely related to both   Schr\"{o}dinger Lie algebras and the Virasoro  algebras, which are known to be important in many areas of mathematics and physics.  Unterberger \cite{unt} constructed the explicit non-trivial vertex algebra representations of the original sector.
In \cite{jw15}  the derivation algebra and the automorphism
group of the original deformative Schr\"{o}dinger-Witt algebras $L_{\lambda,\mu}$ were described. Moreover,
 the second cohomology group of $L_{\lambda,\mu}$
were determined in   \cite{lsz}.


In this paper,  we first point out if a Lie algebra $L$ is $Z$-graded, then the space consisting of all $\frac{1}{2}$-derivations of $L$  is naturally $Z$-graded. With this simple but key observation, we calculate  $\frac{1}{2}$-derivations on central extensions of the extended  Schr\"{o}dinger-Witt  algebras $\widetilde{\mathfrak{so}}$  and  the original deformative Schr\"{o}dinger-Witt algebras $L_{\lambda,\mu}$ by directly calculating any homogeneous subspace of $\frac{1}{2}$-derivations on them. In addition, we prove that central extension of  $L_{\lambda,\mu}$ admits nontrivial transposed Poisson structures only for $\lambda=1$;
the extended  Schr\"{o}dinger-Virasoro   algebras do not admit nontrivial  $\frac{1}{2}$-derivations, also have no  nontrivial transposed Poisson structures.



\section{Preliminaries}\ \ \ \ \ \ \ \ \

In this section, we recall some definitions and known results for studying transposed Poisson structures. Although all algebras and vector spaces are considered over the complex field, many results can be proven over other fields without modifications of proofs.

\begin{definition}
Let $\mathfrak {L}$ be a vector space equipped with two nonzero bilinear operations $\cdot$ and $[\cdot, \cdot]$. The triple $(\mathfrak {L}, \cdot, [\cdot, \cdot])$ is called a transposed Poisson algebra if $(\mathfrak {L}, \cdot)$ is a commutative associative algebra and $(\mathfrak {L}, [\cdot, \cdot])$ is a Lie algebra that satisfies the following compatibility condition
$$
\begin{array}{c}
2z \cdot [x,y]=[z \cdot x, y]+[x, z \cdot y].
\end{array}
$$
\end{definition}

\begin{definition}
Let $(\mathfrak{L}, [\cdot, \cdot])$ be a Lie algebra. A transposed Poisson structure on $(\mathfrak {L}, [\cdot, \cdot])$ is a commutative associative multiplication $\cdot$ in $\mathfrak {L}$ which makes $(\mathfrak {L}, \cdot, [\cdot, \cdot])$ a transposed Poisson algebra.
\end{definition}

\begin{definition}
Let $(\mathfrak {L}, [\cdot, \cdot])$ be a Lie algebra, $\varphi: \mathfrak {L}\rightarrow \mathfrak {L} $ be a linear map. Then $\varphi$ is a $\frac 12$-derivation if it satisfies
$$
\begin{array}{c}
\varphi ([x,y])= \frac 12 \big([\varphi (x), y]+[x, \varphi (y)]\big).
\end{array}
$$
\end{definition}

Observe that $\frac 12$-derivations are a particular case of $\delta$-derivations introduced by Filippov in 1998 \cite{fil} and
recently the notion of $\frac{1}{2}$-derivations of algebras was generalized to
 $\frac{1}{2}$-derivations from an algebra to a module \cite{zz}.
The main example of $\frac 12$-derivations is the multiplication by an element from the ground field. Let us call such $\frac 12$-derivations as  {trivial
$\frac 12$-derivations.} It is easy to see that $[\mathfrak {L},\mathfrak {L}]$ and $\operatorname{Ann}(\mathfrak {L})$ are invariant under any $\frac 12$-derivation of $\mathfrak {L}$.

Let $G$ be an abelian group, $\mathfrak{L}=\bigoplus \limits_{g\in G}\mathfrak{L}_{g}$ be a $G$-graded Lie algebra.
We say that a $\frac 12$-derivation $\varphi$ has degree $g$ ($deg(\varphi)=g$) if $\varphi(\mathfrak{L}_{h})\subseteq \mathfrak{L}_{g+h}$.
Let $\triangle(\mathfrak{L})$ denote the space of $\frac 12$-derivations and write $\triangle_{g}(\mathfrak{L})=\{\varphi \in \triangle(\mathfrak{L}) \mid \deg(\varphi)=g\}$.
The following trivial lemmas are useful in our work.

\begin{lemma}
Let $\mathfrak{L}=\bigoplus \limits_{g\in G}\mathfrak{L}_{g}$ be a $G$-graded Lie algebra. Then
$
\triangle(\mathfrak{L})=\bigoplus \limits_{g\in G}\triangle_{g}(\mathfrak{L}).
$
\end{lemma}

\begin{lemma} \label{l01}(see {\rm\cite{FKL}})
Let $(\mathfrak {L}, \cdot, [\cdot, \cdot])$ be a transposed Poisson algebra and $z$ an arbitrary element from $\mathfrak {L}$.
Then the left multiplication $L_{z}$ in the   commutative associative algebra $(\mathfrak {L}, \cdot)$ gives a $\frac 12$-derivation of the Lie algebra $(\mathfrak {L}, [\cdot, \cdot])$.
\end{lemma}

\begin{lemma}  (see {\rm\cite{FKL}})
Let $\mathfrak {L}$ be a Lie algebra without non-trivial $\frac 12$-derivations. Then every transposed Poisson structure defined on $\mathfrak {L}$ is trivial.
\end{lemma}

\section{Transposed Poisson   structures on the extended Schr\"{o}dinger-Virasoro algebras}\ \ \ \ \ \ \ \ \

Unterberger \cite{unt} introduced a class of new infinite-dimensional Lie algebras $\widetilde{\mathfrak{so}}$ called the extended Schr\"{o}dinger-Witt algebra.
\begin{definition}
The extended    Schr\"{o}dinger-Witt Lie algebra   $\widetilde{\mathfrak{so}}$ is a vector space spanned by a basis  $\{L_n, M_n, N_n, Y_{n+\frac{1}{2}} \ | \ n \in \mathbb{Z}\}$ with the following brackets:
\begin{longtable}{lllll}
     $[L_m,L_n]=(n-m)L_{m+n},$&$[L_m,M_n]=nM_{m+n},$\\
     $ [L_m,N_n]=nN_{m+n},$ & $[N_m,M_n]=2M_{m+n},$\\
     $ [L_m,Y_{n+\frac{1}{2}}]=(n+\frac{1-m}{2})Y_{m+n+\frac{1}{2}},$ &$[N_m,Y_{n+\frac{1}{2}}]=Y_{m+n+\frac{1}{2}},$&$[Y_{m+\frac{1}{2}},Y_{n+\frac{1}{2}}]=(m-n)M_{m+n+1},$\\
     \end{longtable}
for all $m,n\in \mathbb{Z}.$
\end{definition}

In \cite{yh21} the authors computed $\frac{1}{2}$-derivations on extended Schr\"{o}dinger-Witt algebras.
\begin{theorem}(see \cite{yh21})
  Every   $\frac{1}{2}$-derivation on $\widetilde{\mathfrak{so}}$ is trivial.
\end{theorem}

In this section we consider a central extension of the extended Schr\"{o}dinger-Witt algebra $\widetilde{\mathfrak{so}}.$ In \cite{gwp} it is referred  that $\widetilde{\mathfrak{so}}$ has only three independent classes of central extensions.
Let $\hat{\mathfrak{so}}=\widetilde{\mathfrak{so}}\oplus\mathbb{C}C_L\oplus\mathbb{C}C_{LN}\oplus\mathbb{C}C_N$ be the vector space over the complex field $\mathbb{C}$ with a basis $\{L_n, M_n, N_n,Y_{n+\frac{1}{2}}, C_L, C_{LN}, C_N \ | \ n \in \mathbb{Z}\}$ satisfying the following relations

\begin{longtable}{lllll}
     $[L_m,L_n]=(n-m)L_{m+n}+\delta_{m+n,0}\frac{m^3-m}{12}C_L,$&$[L_m,M_n]=nM_{m+n},$\\
     $ [L_m,N_n]=nN_{m+n}+\delta_{m+n,0}(m^2-m)C_{LN},$ & $[N_m,M_n]=2M_{m+n},$\\
     $ [L_m,Y_{n+\frac{1}{2}}]=(n+\frac{1-m}{2})Y_{m+n+\frac{1}{2}},$ &$[N_m,Y_{n+\frac{1}{2}}]=Y_{m+n+\frac{1}{2}},$\\
   $[Y_{m+\frac{1}{2}},Y_{n+\frac{1}{2}}]=(m-n)M_{m+n+1},$&
 $[N_m,N_n]=n\delta_{m+n,0}C_{N},$\\

     \end{longtable}
for all $m,n\in \mathbb{Z}.$ The infinite-dimensional Lie algebra $\hat{\mathfrak{so}}$ considered in this paper called \textit{the extended Schr\"{o}dinger-Virasoro algebra}.
Denote
$$H=\bigoplus_{n\in\mathbb{Z}}\mathbb{C}N_n\oplus\mathbb{C}C_N, \quad \textbf{Vir}=\bigoplus_{n\in\mathbb{Z}}\mathbb{C}L_n\oplus\mathbb{C}C_L, \quad \mathcal{HV}=H\oplus \textbf{Vir}\oplus\mathbb{C}C_{LN},$$
$$\mathcal{S}=\bigoplus_{n\in\mathbb{Z}}\mathbb{C}M_n\bigoplus_{n\in\mathbb{Z}}\mathbb{C}Y_{n+\frac{1}{2}}, \quad \mathcal{H_S}=H\oplus\mathcal{S}.$$
They are  subalgebras of $\hat{\mathfrak{so}},$ where
$H$ is an infinite-dimensional Heisenberg algebra, $\textbf{Vir}$ is the classical Virasoro algebra, $\mathcal{HV}$
is the twisted Heisenberg-Virasoro algebra, $\mathcal{S}$ is a two-step nilpotent Lie algebra and $\mathcal{H_S}$
is the semi-direct product of $H$ and $\mathcal{S}.$ Then
$\hat{\mathfrak{so}}$ is the semi-direct product of the twisted Heisenberg-Virasoro algebra $\mathcal{HV}$ and $\mathcal{S},$
 and $\mathcal{S}$ is an ideal of $\hat{\mathfrak{so}}.$

The description of $\frac{1}{2}$-derivations on Lie algebra  $\mathcal{HV}$ is given in the following theorem.
\begin{theorem} (see \cite{yh21})\label{tr}
 Every   $\frac{1}{2}$-derivation  on $\mathcal{HV}$ is trivial.
\end{theorem}

There is  a $\frac{1}{2}\mathbb{Z}$-grading on $\hat{\mathfrak{so}}$ by
$$\hat{\mathfrak{so}}_0= \langle L_0,M_0, N_0, C_L,C_{LN}, C_N\rangle$$

$$\hat{\mathfrak{so}}_n= \langle L_n,M_n, N_n\rangle,\  n\neq 0, \quad \hat{\mathfrak{so}}_{n+\frac{1}{2}}=\langle  Y_{n+\frac{1}{2}}\rangle,$$
then
\begin{center}
    $\hat{\mathfrak{so}}=\left(\bigoplus_{n\in \mathbb{Z}}\hat{\mathfrak{so}}_n\right)\bigoplus\left(\bigoplus_{n\in \mathbb{Z}}\hat{\mathfrak{so}}_{n+\frac{1}{2}}\right).$
\end{center}
Hence $\Delta(\hat{\mathfrak{so}})$ has a natural $\frac{1}{2}\mathbb{Z}$-grading, i.e.,
\begin{center}$\Delta(\hat{\mathfrak{so}})=\left(\bigoplus_{n\in \mathbb{Z}}\Delta_n(\hat{\mathfrak{so}})\right)\bigoplus\left(\bigoplus_{n\in \mathbb{Z}}\Delta_{n+\frac{1}{2}}(\hat{\mathfrak{so}})\right).$
\end{center}

Now we will compute $\frac{1}{2}$-derivations of the algebra $\hat{\mathfrak{so}}.$

\begin{lemma}\label{even}
$\Delta_0(\hat{\mathfrak{so}})=\langle {\rm Id} \rangle$ and
 $\Delta_{j \in \mathbb{Z}\setminus\{0\}}(\hat{\mathfrak{so}})=0.$

\end{lemma}

\begin{proof}
Suppose that $\varphi_j\in \Delta_j(\hat{\mathfrak{so}})$    be a homogeneous $\frac{1}{2}$-derivation. In this case, we have
\begin{equation}\label{der11}
\varphi_j(\hat{\mathfrak{so}}_n)\subseteq \hat{\mathfrak{so}}_{n+j}, \quad \varphi_j (\hat{\mathfrak{so}}_{n+\frac{1}{2}})\subseteq \hat{\mathfrak{so}}_{n+j+\frac{1}{2}}.
\end{equation}

\noindent By Theorem \ref{tr} and the relation (\ref{der11}) we can get
\begin{longtable}{lll}
  $\varphi_j(L_n)=\delta_{j,0}\lambda L_{n}+a_{j,n}M_{n+j}, $& & $\varphi_j(N_n)=\delta_{j,0}\lambda N_{n}+b_{j,n}M_{n+j}, $\\
\multicolumn{3}{l}{
$\varphi_j(M_n)=c_{j,n}^1 L_{n+j}+c_{j,n}^2M_{n+j}+c_{j,n}^3 N_{n+j}+\delta_{n+j,0}(c_n^1C_L+c_n^2C_{LN}+c_n^3C_{N}),$} \\
$\varphi_j(C_L)=\delta_{j,0}\lambda C_{L},$ & & $\varphi_j(C_{LN})=\delta_{j,0}\lambda C_{LN},$ \\
$ \varphi_j(Y_{n+\frac{1}{2}})=d_{j,n} Y_{n+j+\frac{1}{2}},$& &
$\varphi_j(C_N)=\delta_{j,0}\lambda C_{N},$\\
\end{longtable}
for  all  $n\in \mathbb{Z},$ and for some $\lambda \in \mathbb{C}.$

\begin{enumerate}
    \item[(1)] If $j\neq0,$ then applying $\varphi_j$ to both side of $[L_{m},L_{n}]=(n-m)L_{m+n}+\delta_{m+n,0}\frac{m^3-m}{12}C_L,$ we obtain
\begin{equation}\label{der2}
2(n-m)a_{j,m+n}=(n+j)a_{j,n}-(m+j)a_{j,m}.
\end{equation}
\noindent
Setting $n=0$ in (\ref{der2}), we get $(j-m)a_{j,m}=ja_{j,0}.$ Then taking $m=j$ in this equation, we derive $a_{j,0}=0.$  Consequently, we  have  $a_{j,m}=0$ for $m\neq j.$ Letting $m=j$ and $n=-j$ in (\ref{der2}), we obtain $a_{j,j}=0$, this shows   $a_{j,m}=0$  for all $m\in\mathbb{Z}.$
Now applying $\varphi_j$ to both side of $[L_m,N_n]=nN_{m+n}+\delta_{m+n,0}(m^2-m)C_{LN},$ we obtain
\begin{equation}\label{der3}
2nb_{j,m+n}=(n+j)b_{j,n}.
\end{equation}
\noindent
Setting $n=0$ in (\ref{der3}), we get $b_{j,0}=0.$ Then taking $m=-n$ in this equation, we derive $b_{j,n}=0$  for $n\neq -j.$  Letting $m=-2j$ and $n=j$ in (\ref{der3}), this implies $b_{j,-j}=0$, this shows   $b_{j,n}=0$  for all $n\in\mathbb{Z}.$ Next applying $\varphi_j$ to both side of $[N_{m},M_{n}]=2M_{m+n},$ it gives
\begin{equation}\label{der13}
c_{j,m+n}^1=0, \ 2c_{j,m+n}^2=c_{j,n}^2,\
4c_{j,m+n}^3=-mc_{j,n}^1,
\end{equation}
\begin{equation}\label{der23}
\delta_{m+n+j,0}\left(4c_{m+n}^1C_L+(4c_{m+n}^2+((n+j)^2-(n+j))c_{j,n}^1)C_{LN}+(4c_{m+n}^3-(n+j)c_{j,n}^3)C_{N}\right)=0
\end{equation}
Taking $m=0$ in (\ref{der13}), we obtain $c_{j,n}^1=c_{j,n}^2=c_{j,n}^3=0$
for all $n\in\mathbb{Z}.$ Setting $m=0,$ $n=-j$ in (\ref{der23})
it gives $c_{-j}^1=c_{-j}^2=c_{-j}^3=0.$ Then applying
 $\varphi_j$ to both side of the multiplication $[Y_{m+\frac{1}{2}},Y_{n+\frac{1}{2}}]=(m-n)M_{m+n+1},$ we have $$(m-n+j)d_{j,m}+(m-n-j)d_{j,n}=0.$$ Putting $n=m-j$ in this equation, we deduce $d_{j,m}=0$  for all $m\in\mathbb{Z}.$
It  proves $\varphi_{j}=0$ for $j\in{\mathbb{Z}\setminus\{0\}}.$


  \item[(2)] If $j=0,$ then applying $\varphi_0$ to both side of $[L_{m},L_{n}]=(n-m)L_{m+n}+\delta_{m+n,0}\frac{m^3-m}{12}C_L,$ we obtain (\ref{der2}) for $j=0.$
Taking $m=0$ in (\ref{der2}), it gives $a_{0,n}=0$ for $n\neq 0.$ Setting $m=-n\neq 0$ in (\ref{der2}), we have $a_{0,0}=0.$ Now applying $\varphi_0$ to both side of $[L_m,N_n]=nN_{m+n}+\delta_{m+n,0}(m^2-m)C_{LN},$ we obtain (\ref{der3}) for $j=0.$
Setting $m=0$ in (\ref{der3}), we get $b_{0,n}=0$ for $n\neq0.$ Then taking $m=-n\neq0$ in this equation, we derive $b_{0,0}=0.$
Next applying $\varphi_0$ to both side of $[N_{m},M_{n}]=2M_{m+n},$ it gives
\begin{equation}\label{der33}
c_{0,m+n}^1=0, \ 2c_{0,m+n}^2=c_{0,n}^2+\lambda,\
4c_{0,m+n}^3=-mc_{0,n}^1,
\end{equation}
and the relation (\ref{der23}) for $j=0.$ Taking $m=0$ in
(\ref{der33}), this implies $c_{0,n}^1=c_{0,n}^3=0$  and $c_{0,n}^2=\lambda$
for all $n\in\mathbb{Z}.$  Setting $m=-n$ in (\ref{der23})
it gives $c_{0}^1=c_{0}^2=c_{0}^3=0.$
Similarly, applying
 $\varphi_0$ to both side of the multiplication $[Y_{m+\frac{1}{2}},Y_{n+\frac{1}{2}}]=(m-n)M_{m+n+1},$ we have $$(m-n)(d_{0,m}+d_{0,n})=2(m-n)\lambda.$$
 Putting $m=0$ in this equation, we deduce $d_{0,n}=2\lambda-d_{0,0}$  for all $n\in\mathbb{Z}.$ So we can get  $d_{0,n}=d_{0,0}=\lambda$ for $n\in\mathbb{Z}.$
Hence $\varphi_0=\lambda \ {\rm Id}.$


\end{enumerate}
\end{proof}
\begin{lemma}\label{lemm10}
$\Delta_{j+\frac{1}{2}}(\hat{\mathfrak{so}})=0.$
\end{lemma}

\begin{proof}
Let $\varphi_{j+\frac{1}{2}}\in\Delta_{j+\frac{1}{2}}(\hat{\mathfrak{so}})$  be a homogeneous $\frac{1}{2}$-derivation. Then we have
\begin{equation}\label{derj3}
\varphi_{j+\frac{1}{2}}(\hat{\mathfrak{so}}_n)\subseteq \hat{\mathfrak{so}}_{n+j+\frac{1}{2}}, \quad \varphi_{j+\frac{1}{2}} (\hat{\mathfrak{so}}_{n+\frac{1}{2}})\subseteq \hat{\mathfrak{so}}_{n+j+1}.
\end{equation}
\noindent
By (\ref{derj3}) we can assume that
\begin{center}$\varphi_{j+\frac{1}{2}}(L_n)=\alpha_{j,n}Y_{n+j+\frac{1}{2}}, \quad \varphi_{j+\frac{1}{2}}(N_n)=\beta_{j,n}Y_{n+j+\frac{1}{2}}, \quad
\varphi_{j+\frac{1}{2}}(M_n)=\gamma_{j,n}Y_{n+j+\frac{1}{2}},$\end{center}
$$\varphi_{j+\frac{1}{2}}(Y_{n+\frac{1}{2}})=\sigma_{j,n}L_{n+j+1}+\mu_{j,n}N_{n+j+1}+\tau_{j,n}M_{n+j+1}.$$
where $\alpha_{j,n},  \ \beta_{j,n},\  \gamma_{j,n}, \ \sigma_{j,n}, \ \mu_{j,n},\  \tau_{j,n} \in \mathbb{C}.$

Let us say $\varphi=\varphi_{j+\frac{1}{2}}.$
The algebra $\hat{\mathfrak{so}}$ admits a
$\mathbb{Z}_2$-grading.
Namely,
$$\hat{\mathfrak{so}}_{\overline{0}}=\langle \hat{\mathfrak{so}}_j\rangle_{j \in \mathbb Z} \quad \text{and} \quad \hat{\mathfrak{so}}_{\overline{1}}=\langle \hat{\mathfrak{so}}_{j+\frac{1}{2}}\rangle_{j \in \mathbb Z}.$$

The mapping $\varphi$ changes the grading components.
It is known, that the commutator of one derivation and one $\frac{1}{2}$-derivation gives a new $\frac{1}{2}$-derivation.
Hence,
$[\varphi, {\rm ad}_{Y_{n+\frac{1}{2}}}]$ is a $\frac{1}{2}$-derivation which preserve the grading components.
Namely, it is a $\frac 12$-derivation, described in Lemma  \ref{even},
i.e., $[\varphi, {\rm ad}_{Y_{n+\frac{1}{2}}}]= \alpha_n {\rm Id}.$

It is easy to see, that
$$\alpha_n M_m =
[\varphi, {\rm ad}_{Y_{n+\frac{1}{2}}}](M_m)=
\varphi([M_m, Y_{n+\frac{1}{2}}])-[\varphi(M_m), Y_{n+\frac{1}{2}}]$$
$$=-\gamma_{j,m}[Y_{m+j+\frac{1}{2}}, Y_{n+\frac{1}{2}}]=
-(m+j-n)\gamma_{j,m} M_{m+n+j+1}.$$

For fixed elements $m$ and $j$, we
can choose an element $n,$ such that $ n\neq m+j$ and $n\neq -1-j.$
The next observations give $\gamma_{j,m}=0$ for each $(j,m) \in \mathbb Z \times \mathbb Z.$ Hence,  $\alpha_n=0$ for all $n \in \mathbb Z.$ Then we consider
$$0=
[\varphi, {\rm ad}_{Y_{n+\frac{1}{2}}}](N_m)=
\varphi([N_m, Y_{n+\frac{1}{2}}])-[\varphi(N_m), Y_{n+\frac{1}{2}}]$$
$$=\varphi(Y_{m+n+\frac{1}{2}})-\beta_{j,m}[Y_{m+j+\frac{1}{2}},Y_{n+\frac{1}{2}}]=
\sigma_{j,m+n}L_{m+n+j+1}+\mu_{j,m+n}N_{m+n+j+1}$$
$$+\tau_{j,m+n}M_{m+n+j+1}-\beta_{j,m}(m-n+j)M_{m+n+j+1}.$$

From this we have $\sigma_{j,m}=\mu_{j,m}=0$ and
\begin{equation}\label{der43}\tau_{j,m+n}-\beta_{j,m}(m-n+j)=0
\end{equation}
 for all $m \in \mathbb Z.$ Taking $m=0$ in (\ref{der43}), then  setting $n=0$  we can get
$$\beta_{j,m}=\frac{j-m}{j+m}\beta_{j,0}, \quad m\neq -j.$$
By the relation (\ref{der43}) we obtain $\beta_{j,0}=0,$ it gives
$\tau_{j,m}=\beta_{j,m}=0$ for each $(j,m) \in \mathbb Z \times \mathbb Z.$

$$0=
[\varphi, {\rm ad}_{Y_{n+\frac{1}{2}}}](L_m)=
\varphi([L_m, Y_{n+\frac{1}{2}}])-[\varphi(L_m), Y_{n+\frac{1}{2}}]$$
$$=(n+\frac{1-m}{2})\varphi(Y_{m+n+\frac{1}{2}})-\alpha_{j,m}[Y_{m+j+\frac{1}{2}},Y_{n+\frac{1}{2}}]=-\alpha_{j,m}(m-n+j)M_{m+n+j+1},$$
we can choose an element $n,$ such that $ n\neq m+j.$ It shows that $\alpha_{j,m}=0$ for each $(j,m) \in \mathbb Z \times \mathbb Z,$ summarizing, $\varphi=0.$
 \end{proof}

Summarizing the results from lemmas \ref{even} and \ref{lemm10}, we conclude that $\hat{\mathfrak{so}}$  does not have nontrivial $\frac{1}{2}$-derivations.
\begin{theorem}
 $\hat{\mathfrak{so}}$ has no nontrivial  $\frac{1}{2}$-derivations.
\end{theorem}

\begin{corollary}
 $\hat{\mathfrak{so}}$ has no nontrivial  transposed Poisson algebra structures.
\end{corollary}

\section{Transposed Poisson   structures on original deformative  Schr\"{o}dinger-Virasoro algebras}\ \ \ \ \ \ \ \ \

The infinite-dimensional  original deformative  Schr\"{o}dinger-Witt  algebras were considered in the  paper \cite{lsz} and  denoted by $L_{\lambda,\mu}$ $(\lambda,\mu\in \mathbb{C})$, possess   the basis
$\{L_n, M_n, Y_{n+\frac{1}{2}} \ | \ n\in\mathbb{Z} \}$
with the following non-vanishing  Lie brackets:
\begin{longtable}{lllll}
     $[L_m,L_n]=(n-m)L_{m+n},$&$ [L_m,M_n]=(n-\lambda m+2\mu)M_{m+n},$\\
     $ [L_m,Y_{n+\frac{1}{2}}]=(n+\frac{1}{2}-\frac{\lambda+1}{2}m+\mu)Y_{m+n+\frac{1}{2}}, $& $[Y_{m+\frac{1}{2}},Y_{n+\frac{1}{2}}]=(n-m)M_{m+n+1}.$\\
     \end{longtable}

From \cite{lsz} it is known that the original deformative Schr\"{o}dinger-Witt  algebras have central extensions and it can be formulated follows:

\begin{scriptsize}

\[
\begin{array}{ll}
      \widetilde{L_{\lambda,\mu}}^1  : \left\{ \begin{array}{l}
         [L_m,L_n]=(n-m)L_{m+n}+\frac{m^3-m}{12}\delta_{m+n,0}C_L, \\

     [L_m,Y_{n+\frac{1}{2}}]=(n+\frac{1}{2}-\frac{\lambda+1}{2}m+\mu)Y_{m+n+\frac{1}{2}},\\

          [L_m,M_n]=(n-\lambda m+2\mu)M_{m+n,}  \\

          [Y_{m+\frac{1}{2}},Y_{n+\frac{1}{2}}]=(n-m)M_{m+n+1},
    \end{array} \right.
    &

    \widetilde{L_{\lambda,\mu}}^2 : \left\{ \begin{array}{l}
         [L_m,L_n]=(n-m)L_{m+n}+\frac{m^3-m}{12}\delta_{m+n,0}C_L, \\

      [L_m,Y_{n+\frac{1}{2}}]=(n+\frac{1}{2}-\frac{\lambda+1}{2}m+\mu)Y_{m+n+\frac{1}{2}}+\\
      \delta_{m+n+\mu+\frac{1}{2},0}C_{LY},  \\

          [L_m,M_n]=(n-\lambda m+2\mu)M_{m+n,}  \\

          [Y_{m+\frac{1}{2}},Y_{n+\frac{1}{2}}]=(n-m)M_{m+n+1},

    \end{array} \right.

    \\

\end{array}
\]

\[
\begin{array}{ll}
     \widetilde{L_{\lambda,\mu}}^3 : \left\{ \begin{array}{l}
        [L_m,L_n]=(n-m)L_{m+n}+\frac{m^3-m}{12}\delta_{m+n,0}C_L, \\

     [L_m,Y_{n+\frac{1}{2}}]=(n+\frac{1}{2}-\frac{\lambda+1}{2}m+\mu)Y_{m+n+\frac{1}{2}}+\\ \frac{m(m-1)}{2}\delta_{m+n+\mu+\frac{1}{2},0}C_{LY},\\

          [L_m,M_n]=(n-\lambda m+2\mu)M_{m+n,}  \\

          [Y_{m+\frac{1}{2}},Y_{n+\frac{1}{2}}]=(n-m)M_{m+n+1},\\

          [M_m,Y_{n+\frac{1}{2}}]=\delta_{m+n+3\mu+\frac{1}{2},0}C_{MY},
    \end{array} \right.
    &

    \widetilde{L_{\lambda,\mu}}^4 : \left\{\begin{array}{l}
         [L_m,L_n]=(n-m)L_{m+n}+\frac{m^3-m}{12}\delta_{m+n,0}C_L, \\

    [L_m,Y_{n+\frac{1}{2}}]=(n+\frac{1}{2}-\frac{\lambda+1}{2}m+\mu)Y_{m+n+\frac{1}{2}}-\\
    m(m^2-1)\delta_{m+n+\mu+\frac{1}{2},0}C_{LY},\\

          [L_m,M_n]=(n-\lambda m+2\mu)M_{m+n}-\\
          m(m^2-1)\delta_{m+n+2\mu,0}C_{M},  \\

          [Y_{m+\frac{1}{2}},Y_{n+\frac{1}{2}}]=(n-m)M_{m+n+1}-\\
          (m+\mu)((m+\mu)^2-1)\delta_{m+n+2\mu,0}C_{M},
    \end{array} \right.
\end{array}
\]
\[
\begin{array}{ll}

    \widetilde{L_{\lambda,\mu}}^5 : \left\{ \begin{array}{l}
         [L_m,L_n]=(n-m)L_{m+n}+\frac{m^3-m}{12}\delta_{m+n,0}C_L, \\

     [L_m,Y_{n+\frac{1}{2}}]=(n+\frac{1}{2}-\frac{\lambda+1}{2}m+\mu)Y_{m+n+\frac{1}{2}},\\

          [L_m,M_n]=(n-\lambda m+2\mu)M_{m+n},  \\

         [Y_{m+\frac{1}{2}},Y_{n+\frac{1}{2}}]=(n-m)M_{m+n+1}-(m+\mu+\frac{1}{2})\delta_{m+n+2\mu+1,0}C_{Y},
    \end{array} \right.
 &

\end{array}
\]
\end{scriptsize}
where

$\widetilde{L_{\lambda,\mu}}^1$:  $\mu\notin\{\frac{1}{2}\mathbb{Z}\}$ or $\mu\in\frac{1}{2}+\mathbb{Z}$ and $\lambda\neq-3,-1,1$  or $\mu\in\mathbb{Z}$ and $\lambda\neq -1;$

$\widetilde{L_{\lambda,\mu}}^2$:  $\mu\in\frac{1}{2}+\mathbb{Z}$ and $\lambda=-3;$

$\widetilde{L_{\lambda,\mu}}^3$: $\mu\in\frac{1}{2}+\mathbb{Z}$ and $\lambda=-1;$

$\widetilde{L_{\lambda,\mu}}^4$: $\mu\in\frac{1}{2}+\mathbb{Z}$ and $\lambda=1;$

$\widetilde{L_{\lambda,\mu}}^5$:  $\mu\in\mathbb{Z},$ and $\lambda= -1.$






In this work, the algebras denoted by $\widetilde{L_{\lambda,\mu}}^i$  ($i=\overline{1,5}$) will be referred to as \textit{the original deformative
Schr\"{o}dinger-Virasoro algebras}.   Now, we compute $\frac{1}{2}$-derivations on $\widetilde{L_{\lambda,\mu}}^i$  ($i=\overline{1,5}$). We begin by  calculating $\frac{1}{2}$-derivations on $\widetilde{L_{\lambda,\mu}}^1.$

There is  a $\frac{1}{2}\mathbb{Z}$-grading on $\widetilde{L_{\lambda,\mu}}^1$ by
$$W_0= \langle L_0,M_0,  C_L\rangle, \quad W_n= \langle L_n, M_n\rangle, \  n\neq 0, \quad W_{n+\frac{1}{2}}=\langle  Y_{n+\frac{1}{2}}\rangle,$$
then
\begin{center}
    $\widetilde{L_{\lambda,\mu}}^1=\left(\bigoplus_{n\in \mathbb{Z}}W_n\right)\bigoplus\left(\bigoplus_{n\in \mathbb{Z}}W_{n+\frac{1}{2}}\right).$
\end{center}
Hence $\Delta(\widetilde{L_{\lambda,\mu}}^1)$ has a natural $\frac{1}{2}\mathbb{Z}$-grading, i.e.,
\begin{center}$\Delta(\widetilde{L_{\lambda,\mu}}^1)=\left(\bigoplus_{n\in \mathbb{Z}}\Delta_n(\widetilde{L_{\lambda,\mu}}^1)\right)\bigoplus\left(\bigoplus_{n\in \mathbb{Z}}\Delta_{n+\frac{1}{2}}(\widetilde{L_{\lambda,\mu}}^1)\right).$
\end{center}

\begin{lemma}\label{even1}
\begin{enumerate}
\item If $\mu\in\frac{1}{2}+\mathbb{Z}$ and $\lambda\neq-3,-1,1,$ then
   $\Delta_j(\widetilde{L_{\lambda,\mu}}^1)$  is trivial;
\item if $\mu\notin\{\frac{1}{2}\mathbb{Z}\}$ or  $\mu\in\mathbb{Z}$ and $\lambda\neq -1,$ then
 $\Delta_j(\widetilde{L_{\lambda,\mu}}^1)=\langle \operatorname{Id}\rangle\,$  for $\lambda\neq 1$
 and $\Delta_j(\widetilde{L_{1,\mu}}^1)=\langle \operatorname{Id},\varphi_j\rangle\,$ where $\varphi_j(L_n)=\alpha_j M_{n+j} $   for all $n\in\mathbb{Z}.$
\end{enumerate}
\end{lemma}

\begin{proof}
Suppose that $\varphi_j\in \Delta_j(\widetilde{L_{\lambda,\mu}}^1))$    be a homogeneous $\frac{1}{2}$-derivation. In this case, we have
\begin{equation}\label{der}
\varphi_j(W_n)\subseteq W_{n+j}, \quad \varphi_j (W_{n+\frac{1}{2}})\subseteq W_{n+j+\frac{1}{2}}.
\end{equation}

\noindent Note that $\widetilde{L_{\lambda,\mu}}^1$ contains a subalgebra $\langle  L_m \ |\ m\in\mathbb{Z}\rangle,$ which isomorphic to the well-known algebra $\textbf{Vir}$ and every $\frac{1}{2}$-derivation on Virasoro algebra is trivial \cite{FKL}.
By (\ref{der}), we can assume that
\begin{center}$\varphi_{j}(L_n)=\delta_{j,0}\lambda_1 L_{n}+\alpha_{j,n}M_{n+j}, \quad \varphi_{j}(Y_{n+\frac{1}{2}})=\sigma_{j,n}Y_{n+j+\frac{1}{2}},$
\end{center}
\begin{center} $\varphi_{j}(M_n)=\beta_{j,n}L_{n+j}+\gamma_{j,n}M_{n+j}+\delta_{n+j,0}c_nC_L, \quad
\varphi_{j}(C_L)=\delta_{j,0}\lambda_1 C_{L},$
\end{center}
where $\alpha_{j,n},  \ \beta_{j,n},\  \gamma_{j,n}, \ \sigma_{j,n}, \ c_{n}\  \in \mathbb{C}.$

Now, we start with applying $\varphi_j$ to both side of
$[Y_{m+\frac{1}{2}},Y_{n+\frac{1}{2}}]=(n-m)M_{m+n+1},$
and we have
\begin{eqnarray}
 &&(n-m)\beta_{j,m+n+1}=0,\label{ad}\\
 &&(n-m)\delta_{m+n+j+1,0}c_{m+n+1}=0,\label{ac}\\
&&2(n-m)\gamma_{j,m+n+1}=(n-m-j)\sigma_{j,m}+(n+j-m)\sigma_{j,n}.\label{bc}
\end{eqnarray}

Putting $n=-1$ in (\ref{ad}), we get
$\beta_{j,m}=0$ for $m\neq -1,$ then taking $m=0,$ $n=-2$, we have $\beta_{j,-1}=0,$ which follows $\beta_{j,m}=0$ for all $m\in\mathbb{Z}.$ From  (\ref{ac}),
we deduce $c_{-j}=0$ for all $j\in\mathbb{Z}.$ Next,  applying
 $\varphi_j$ to $[L_m,L_n]=(n-m)L_{m+n}+\frac{m^3-m}{12}\delta_{m+n,0}C_L$
and $ [L_m,M_n]=(n-\lambda m+2\mu)M_{m+n},$ we have  equations
\begin{equation}\label{bd}
 2(n-m)\alpha_{j,m+n}=(n+j-\lambda m+2\mu)\alpha_{j,n}+(m+j-\lambda n+2\mu)\alpha_{j,m},
 \end{equation}
 \begin{equation}\label{ab}
 2(n-\lambda m+2\mu)\gamma_{j,m+n}=\delta_{j,0}\lambda_1(n-\lambda m+2\mu)+(n+j-\lambda m+2\mu)\gamma_{j,n}.
 \end{equation}

Taking $m=0$ in (\ref{bd}), we obtain $(j-n+2\mu)\alpha_{j,n}=(j-\lambda n+2\mu)\alpha_{j,0}$ for all $n\in \mathbb{Z}.$ If $\lambda=1,$ then we have $\alpha_{j,n}=\alpha_{j,0}$ for all
$n \in \mathbb Z.$ If $\lambda\neq1,$ then it follows
\begin{equation}\label{da}
\alpha_{j,n}=\frac{(j-\lambda n+2\mu)}{(j-n+2\mu)}\alpha_{j,0}, \quad 2\mu\notin \mathbb{Z}.
\end{equation}
Then using (\ref{da}) in  (\ref{bd}), we can get
\begin{equation}\label{dc}
\frac{mn(\lambda-1)(n-m)(2\mu(\lambda+2)-(m+n-j)\lambda+2j)}{(j-m+2\mu)(j-n+2\mu)(j-m-n+2\mu)}\alpha_{j,0}=0, \quad 2\mu\notin \mathbb{Z}.
\end{equation}
Due to  arbitrary $m$ in (\ref{dc}), it gives that
$\alpha_{j,0}=0,$ and from (\ref{da}) it follows $\alpha_{j,n}=0$ for arbitrary $n\in \mathbb Z.$
If $2\mu\in \mathbb{Z},$ then it shows $\alpha_{j,n}=0$ for $n\neq j+2\mu.$
Putting  $m\neq 0$ and $n= j+2\mu$ in (\ref{bd}), it gives $\alpha_{j,j+2\mu}=0,$ so we derive  $\alpha_{j,n}=0$ for all  $n\in \mathbb{Z}.$

\begin{enumerate}
    \item[(1)] If $j\neq0,$ then letting $m=0$ in  (\ref{ab}), we obtain
\begin{equation}\label{derk2}
(n-j+2\mu)\gamma_{j,n}=0.
\end{equation}
\noindent
If $2\mu\notin \mathbb{Z},$ then $\gamma_{j,n}=0$ for all $n\in \mathbb{Z},$ if $2\mu\in \mathbb{Z},$ then $\gamma_{j,n}=0$ for  $n\neq j-2\mu.$ Setting $m\neq 0$ and $n= j-2\mu$ in (\ref{derk2}), it gives $\gamma_{j,j-2\mu}=0,$ which derive  $\gamma_{j,n}=0$ for all  $n\in \mathbb{Z}.$ Next letting $m=n-j$
in  (\ref{bc}), we obtain $\sigma_{j,n}=0$ for all  $n\in \mathbb{Z}.$

 \item[(2)] If $j=0,$ then using (\ref{bc}) and (\ref{ab}), we get
 $\gamma_{0,n}=\sigma_{0,n}=\lambda_1$ for all  $n\in\mathbb{Z}.$

 Hence,  $$\varphi_j(L_n)=\alpha_{j,0}M_{n+j},  \quad \textrm{for} \quad \lambda= 1,$$
 $$\varphi_j=0, \ j\neq 0, \quad \varphi_0=\lambda_1 \ {\rm Id}, \quad \textrm{for} \quad \lambda\neq 1.$$
\end{enumerate}
\end{proof}

\begin{lemma}\label{lem10}
\begin{enumerate}
\item If $\mu\in\frac{1}{2}+\mathbb{Z}$ and $\lambda\neq-3,-1,1,$ then
   $\Delta_{j+\frac{1}{2}}(\widetilde{L_{\lambda,\mu}}^1)=0$ for all $j\in\mathbb{Z}$ ;
\item if $\mu\notin\{\frac{1}{2}\mathbb{Z}\}$ or  $\mu\in\mathbb{Z}$ and $\lambda\neq -1,$ then
$\Delta_{j+\frac{1}{2}}(\widetilde{L_{\lambda,\mu}}^1)=0$ for $\lambda\neq1$ and
 $$\Delta_{j+\frac{1}{2}}(\widetilde{L_{1,\mu}}^1)=\{\varphi_{j+\frac{1}{2}} \ | \ \varphi_{j+\frac{1}{2}}(L_n)=\alpha_j Y_{n+j+\frac{1}{2}}, \ \varphi_{j+\frac{1}{2}}( Y_{n+\frac{1}{2}})=\alpha_jM_{n+j+1}\} $$   for all $n\in\mathbb{Z}.$
\end{enumerate}
\end{lemma}
\begin{proof}
Let $\varphi_{j+\frac{1}{2}}\in\Delta_{j+\frac{1}{2}}(\widetilde{L_{\lambda,\mu}}^1)$  be a homogeneous $\frac{1}{2}$-derivation. Then we have
\begin{equation}\label{derk3}
\varphi_{j+\frac{1}{2}}(W_n)\subseteq W_{n+j+\frac{1}{2}}, \quad \varphi_{j+\frac{1}{2}} (W_{n+\frac{1}{2}})\subseteq W_{n+j+1}.
\end{equation}
\noindent
By (\ref{derk3}), we can assume that
\begin{center}$\varphi_{j+\frac{1}{2}}(L_n)=\alpha_{j,n}Y_{n+j+\frac{1}{2}}, \quad
\varphi_{j+\frac{1}{2}}(M_n)=\beta_{j,n}Y_{n+j+\frac{1}{2}},$\end{center}
$$\varphi_{j+\frac{1}{2}}(Y_{n+\frac{1}{2}})=\gamma_{j,n}L_{n+j+1}+\mu_{j,n}M_{n+j+1}+\delta_{n+j+1,0}c_{n}C_{L}.$$
where $\alpha_{j,n},  \ \beta_{j,n},\  \gamma_{j,n}, \ \mu_{j,n}, \   c_{n} \in \mathbb{C}.$
Since $\operatorname{Ann}(\widetilde{L_{\lambda,\mu}}^1)$ is invariant under any $\frac{1}{2}$-derivation, it gives $\varphi_{j+\frac{1}{2}}(C_L)=0$ for all $j\in\mathbb{Z}.$

Let us say $\varphi=\varphi_{j+\frac{1}{2}}.$
The algebra $\widetilde{L_{\lambda,\mu}}^1$ admits a
$\mathbb{Z}_2$-grading.
Namely,
$$\widetilde{L_{\lambda,\mu}}^1_{\overline{0}}=\langle W_j\rangle_{j \in \mathbb Z} \quad
\text{and}
 \quad \widetilde{L_{\lambda,\mu}}^1_{\overline{1}}=\langle W_{j+\frac{1}{2}}\rangle_{j \in \mathbb Z}.$$

The mapping $\varphi$ changes the grading components.
It is known, that the commutator of one derivation and one $\frac{1}{2}$-derivation gives a new $\frac{1}{2}$-derivation.
Hence,
$[\varphi, {\rm ad}_{Y_{n+\frac{1}{2}}}]$ is a $\frac{1}{2}$-derivation which preserve the grading components.
Namely, it is a $\frac 12$-derivation, described in Lemma  \ref{even1},
i.e., $[\varphi, {\rm ad}_{Y_{n+\frac{1}{2}}}]= \alpha_n {\rm Id}$ for $\lambda\neq 1.$ Similarly to the proof of the Lemma \ref{lemm10}, we deduce $\varphi=0$ for $\lambda\neq 1.$

Now, we consider the case of $\lambda=1.$  Using  the Lemma \ref{even1}, we can get these relations $[\varphi, {\rm ad}_{Y_{n+\frac{1}{2}}}](M_m)=\alpha_nM_m$ and $[\varphi, {\rm ad}_{Y_{n+\frac{1}{2}}}](Y_{n+\frac{1}{2}})=\alpha_nY_{m+\frac{1}{2}}.$ It gives $\beta_{j,m}=\gamma_{j,m}=0$ for all
$m \in \mathbb {Z}.$ Next applying $\varphi_{j+\frac{1}{2}}$ to both side of $[L_m,L_n]=(n-m)L_{m+n}+\frac{m^3-m}{12}\delta_{m+n,0}C_L,$ it implies
\begin{eqnarray}
(2(n+j)+1-2m+2\mu)\alpha_{j,n}-(2(m+j)+1-2n+2\mu)\alpha_{j,m}=4(n-m)\alpha_{j,m+n}.\label{2.1}\end{eqnarray}
Taking $m=0$ in (\ref{2.1}), we have $\alpha_{j,n}=\alpha_{j,0}$ for  $n \in \mathbb {Z}.$
Then applying $\varphi_{j+\frac{1}{2}}$ to  $ [L_m,Y_{n+\frac{1}{2}}]=(n+\frac{1}{2}-m+\mu)Y_{m+n+\frac{1}{2}}, $ we obtain

\begin{equation}\label{2.2}
 (2n+1-2m+2\mu)\delta_{m+n+j+1,0}c_{m+n}C_L=0,
\end{equation}
\begin{eqnarray}(n-m-j)\alpha_{j,0}+(n+j+1-m+2\mu)\mu_{j,n}=(2n+1-2m+2\mu)\mu_{j,m+n}.\label{2.4}\end{eqnarray}
From (\ref{2.2}), we derive $c_{-j-1}=0$ and similarly to the previous case from (\ref{2.4}) we imply
$\mu_{j,n}=\alpha_{j,0}$ for all $n \in \mathbb {Z}.$
\end{proof}

Summarizing the results from Lemmas \ref{even1} and \ref{lem10}, we conclude that
if $\lambda\neq1,$ then $\widetilde{L_{\lambda,\mu}}^1$ does not have nontrivial $\frac{1}{2}$-derivations.
On the other side,
Theorem \ref{maini2} gives the full description of nontrivial  $\frac{1}{2}$-derivations of $\widetilde{L_{1,\mu}}^1.$

\begin{theorem}\label{maini2}
Let $\varphi$ be a $\frac{1}{2}$-derivation of the algebra $\widetilde{L_{1,\mu}}^1,$ then there are two  sets of elements from the basic field $\{\alpha_t\}_{t\in\mathbb{Z}}$  and $\{\beta_t\}_{t\in\mathbb{Z}}$
such that
\begin{longtable}{llll}
$\varphi(L_m)$ & = &$\lambda_1L_m+\sum\limits_{t\in{\mathbb{Z}}}\alpha_tM_{m+t}+\sum\limits_{t\in{\mathbb{Z}}}\beta_tY_{m+t+\frac{1}{2}},$\\
$\varphi(Y_{m+\frac{1}{2}})$ & = &$\lambda_1Y_{m+\frac{1}{2}}+\sum\limits_{t\in{\mathbb{Z}}}\beta_tM_{m+t+1},$\\
$\varphi(M_m)$ & = &$\lambda_1M_m,$\\
$\varphi(C_L)$ & = &$\lambda_1C_L.$\\
\end{longtable}
\end{theorem}
\begin{proof}
  The proof follows directly from the lemmas \ref{even1} and \ref{lem10}.
\end{proof}


In the following, we aim to classify all transposed Poisson structures on  $\widetilde{L_{1,\mu}}^1.$
\begin{theorem} Let  $(\widetilde{L_{1,\mu}}^1,\cdot,[\cdot,\cdot])$ be a transposed Poisson algebra structure defined on the Lie algebra $\widetilde{L_{1,\mu}}^1.$
 Then the commutative associative multiplication on $(\widetilde{L_{1,\mu}}^1,\cdot)$ has the following form:
\begin{longtable}{lcl}
$L_m\cdot L_n $&$=$&$\sum\limits_{t\in \mathbb{Z}}\alpha_{t}M_{m+n+t}+\sum\limits_{t\in \mathbb{Z}}\beta_{t}Y_{m+n+t+\frac{1}{2}},$\\
$L_m\cdot Y_{n+\frac{1}{2}}$&$=$&$\sum\limits_{t\in \mathbb{Z}}\beta_{t}M_{m+n+t+1},$\\
\end{longtable} where $\alpha_{t},\beta_{t}\in \mathbb{C}$ for all $t\in\mathbb{Z}.$
\end{theorem}
\begin{proof} We aim to describe the multiplication $\cdot.$ By Lemma \ref{l01}, for every element  $X\in\{L_i, M_i, Y_{i+\frac{1}{2}}, C_L \ | \  {i\in\mathbb Z}\},$ there is a related
$\frac{1}{2}$-derivation $\varphi_{X}$ of $\widetilde{L_{1,\mu}}^1,$ such that $\varphi_{X}(Y)= X \cdot Y.$
Then by Theorem \ref{maini2}, we have that
\begin{longtable}{llll}
$\varphi_{X}(L_m)$ & = &$\lambda_{1,X}L_m+\sum\limits_{t\in{\mathbb{Z}}}\alpha_{t, X}M_{m+t}+\sum\limits_{t\in{\mathbb{Z}}}\beta_{t, X}Y_{m+t+\frac{1}{2}},$\\
$\varphi_{X}(Y_{m+\frac{1}{2}})$ & = &$\lambda_{1,X}Y_{m+\frac{1}{2}}+\sum\limits_{t\in{\mathbb{Z}}}\beta_{t,X}M_{m+t+1},$\\
$\varphi_{X}(M_m)$ & = &$\lambda_{1,X}M_m,$\\
$\varphi_{X}(C_L)$ & = &$\lambda_{1,X}C_L.$
\end{longtable}

Now we consider $\varphi_{X}(Y) = X \cdot Y = Y \cdot X = \varphi_{Y}(X)$ for $X,Y \in   \{ L_i, M_i,Y_{i+\frac{1}{2}}, C_L \ | \  {i \in \mathbb Z} \}.$
Firstly, by $\varphi_{X}(C_L)=\lambda_{1,X}C_L$ it follows $$C_L\cdot X=X\cdot C_L=0.$$ Similarly, we can get
$$M_i\cdot X=X\cdot M_i=0$$
for $X \in   \{ L_i, M_i,Y_{i+\frac{1}{2}}, C_L \ | \  {i \in \mathbb Z} \}.$

\begin{enumerate}[I.]





\item Let $X=L_m$ and $Y= L_0.$ Then the equality $\varphi_{L_m}(L_0)=L_m\cdot L_0=L_0\cdot L_m=\varphi_{L_0}(L_m),$ gives
$$\sum\limits_{t\in \mathbb{Z}}\alpha_{t,L_m}M_{t}+\sum\limits_{t\in \mathbb{Z}}\beta_{t,L_m}Y_{t+\frac{1}{2}}=\sum\limits_{t\in \mathbb{Z}}\alpha_{t,L_0}M_{m+t}+\sum\limits_{t\in \mathbb{Z}}\beta_{t,L_0}Y_{m+t+\frac{1}{2}}.$$

Hence, we obtain  $\alpha_{k,L_m}=\alpha_{k-m,L_0}$ and $\beta_{k,L_m}=\beta_{k-m,L_0}.$


\item Let $X=L_0$ and $Y= Y_{n+\frac{1}{2}},$ then from
$\varphi_{L_0}(Y_{n+\frac{1}{2}})=\varphi_{Y_{n+\frac{1}{2}}}(L_0),$ we get
$$\sum_{t\in \mathbb{Z}}\beta_{t,L_0}M_{n+t+1}=\sum\limits_{t\in \mathbb{Z}}\alpha_{t,Y_{n+\frac{1}{2}}}M_{t}+\sum\limits_{t\in \mathbb{Z}}\beta_{t,Y_{n+\frac{1}{2}}}Y_{t+\frac{1}{2}}.$$

Thus, we obtain  $\beta_{k,Y_{n+\frac{1}{2}}}=0,$ $\alpha_{k,Y_{n+\frac{1}{2}}}=\beta_{k-n-1,L_0}.$
\end{enumerate}

Summarizing all the above parts, we have that the multiplication table of $(\widetilde{L_{1,\mu}}^1, \cdot)$ is given by following non-trivial relations.
\begin{longtable}{lcl}
$L_m\cdot L_n $&$=$&$\sum\limits_{t\in \mathbb{Z}}\alpha_{t,L_0}M_{m+n+t}+\sum\limits_{t\in \mathbb{Z}}\beta_{t,L_0}Y_{m+n+t+\frac{1}{2}},$\\
$L_m\cdot Y_{n+\frac{1}{2}}$&$=$&$\sum\limits_{t\in \mathbb{Z}}\beta_{t,L_0}M_{m+n+t+1},$\\
\end{longtable}


It gives the complete statement of the theorem.
\end{proof}
In the following theorem we give a description  $\frac{1}{2}$-derivations on the original deformative  Schr\"{o}dinger-Virasoro algebras $\widetilde{L_{-3,\mu}}^2$, $\widetilde{L_{-1,\mu}}^3$ and $\widetilde{L_{-1,\mu}}^5.$
\begin{theorem}\label{maini4}
  Every   $\frac{1}{2}$-derivation on $\widetilde{L_{-3,\mu}}^2$, $\widetilde{L_{-1,\mu}}^3,$ $\widetilde{L_{1,\mu}}^4$ and $\widetilde{L_{-1,\mu}}^5$ is trivial.
\end{theorem}
\begin{proof}
The proof    is similar to the proof of  Lemmas  \ref{even1} and \ref{lem10}.
\end{proof}

By Theorem \ref{maini4}, the following is straightforward.
\begin{corollary}
The infinite-dimensional Lie algebras $\widetilde{L_{-3,\mu}}^2$, $\widetilde{L_{-1,\mu}}^3,$ $\widetilde{L_{1,\mu}}^4$ and $\widetilde{L_{-1,\mu}}^5$ have  no nontrivial  transposed Poisson algebra structures.
\end{corollary}

\end{document}